\documentclass[reqno,12pt,a4paper]{amsart}

\usepackage[margin=2cm]{geometry}

\usepackage{amsmath,amsthm,amssymb,amsfonts,bbm,dsfont,ifpdf,mathtools,verbatim,yhmath}
\usepackage{mathrsfs}
\usepackage{enumitem}
\usepackage{graphicx}
\usepackage{tikz-cd}
\usepackage[all]{xy}
\usepackage{csquotes}
\MakeOuterQuote{"}
\usepackage[colorlinks=true, linkcolor=blue, urlcolor=blue, citecolor=green]{hyperref}

\usepackage{fancyhdr}
\numberwithin{equation}{section}

\newtheorem {thm}    {Theorem}[section]
\newtheorem {lem}      [thm]    {Lemma}

\newtheorem {prop}[thm]    {Proposition}
\newtheorem* {prop*} {Proposition}

\newtheorem*{claim*}   {Claim}

\newtheorem*{conj*} {Conjecture}

\theoremstyle{definition}

\newtheorem {rmk}    [thm]    {Remark}
\newtheorem*{rmk*}  {Remark}

\newtheorem {qst}   [thm]    {Question}
\newtheorem*{qst*} {Question}

\newtheorem* {problem*}{Problem}

\newcommand{\eq}[1]{\begin{equation*}#1 \end{equation*}}
\newcommand{\eql}[2]{\begin{equation}\label{#1}#2\end{equation}}
\newcommand{\eqs}[1]{\begin{equation*}\begin{split}#1\end{split}\end{equation*}}

\DeclareMathSymbol{A}{\mathalpha}{operators}{`A}%
\DeclareMathSymbol{B}{\mathalpha}{operators}{`B}%
\DeclareMathSymbol{C}{\mathalpha}{operators}{`C}%
\DeclareMathSymbol{D}{\mathalpha}{operators}{`D}%
\DeclareMathSymbol{E}{\mathalpha}{operators}{`E}%
\DeclareMathSymbol{F}{\mathalpha}{operators}{`F}%
\DeclareMathSymbol{G}{\mathalpha}{operators}{`G}%
\DeclareMathSymbol{H}{\mathalpha}{operators}{`H}%
\DeclareMathSymbol{I}{\mathalpha}{operators}{`I}%
\DeclareMathSymbol{J}{\mathalpha}{operators}{`J}%
\DeclareMathSymbol{K}{\mathalpha}{operators}{`K}%
\DeclareMathSymbol{L}{\mathalpha}{operators}{`L}%
\DeclareMathSymbol{M}{\mathalpha}{operators}{`M}%
\DeclareMathSymbol{N}{\mathalpha}{operators}{`N}%
\DeclareMathSymbol{O}{\mathalpha}{operators}{`O}%
\DeclareMathSymbol{P}{\mathalpha}{operators}{`P}%
\DeclareMathSymbol{Q}{\mathalpha}{operators}{`Q}%
\DeclareMathSymbol{R}{\mathalpha}{operators}{`R}%
\DeclareMathSymbol{S}{\mathalpha}{operators}{`S}%
\DeclareMathSymbol{T}{\mathalpha}{operators}{`T}%
\DeclareMathSymbol{U}{\mathalpha}{operators}{`U}%
\DeclareMathSymbol{V}{\mathalpha}{operators}{`V}%
\DeclareMathSymbol{W}{\mathalpha}{operators}{`W}%
\DeclareMathSymbol{X}{\mathalpha}{operators}{`X}%
\DeclareMathSymbol{Y}{\mathalpha}{operators}{`Y}%
\DeclareMathSymbol{Z}{\mathalpha}{operators}{`Z}%

\newcommand {\liml}   {\lim\limits}

\newcommand {\cO} {{\mathcal O}}

\DeclareMathOperator{\SL}{SL}

\DeclareMathOperator{\SO}{SO}

\DeclareMathOperator{\Stab}{Stab}

\newcommand{\eps}{\varepsilon}

\newcommand {\IGNORE}[1]  {}

\newcommand {\bsl} {\backslash}

\newcommand{\bC}{\mathbf{C}}

\newcommand{\bH}{\mathbf{H}}

\newcommand{\bN}{\mathbf{N}}

\newcommand{\bQ}{\mathbf{Q}}
\newcommand{\bR}{\mathbf{R}}
\newcommand{\bZ}{\mathbf{Z}}

\newcommand {\La} {{\Lambda}}

\newcommand\vol{\operatorname{vol}}

\newcommand{\sK}{\mathsf{K}}

\newcommand{\sM}{\mathsf{M}}

\renewcommand{\leq}{\leqslant}
\renewcommand{\geq}{\geqslant}
\DeclareMathOperator{\Mod}{mod\;}
\DeclareMathOperator{\Or}{\scriptstyle\cO}
\renewcommand{\O}{O}
\renewcommand{\o}{o}

\begin{document}
	
	\title[Unboundedness of shapes of unit lattices]{Unboundedness of shapes of unit lattices in totally real cubic fields}
	\author[E.Corso]{Emilio Corso}
	\address{Pennsylvania State University, Department of Mathematics, 203 McAllister Building, State College, PA 16802}
	\email{corsoemilio2@gmail.com}
	\author[F.Rodriguez Hertz]{Federico Rodriguez Hertz}
	\address{Pennsylvania State University, Department of Mathematics, 328 McAllister Building, State College, PA 16802}
	\email{frj@psu.edu}
	\date{\today}
	
	
	\begin{abstract}
		The question of the distribution of shapes of unit lattices in number fields, pioneered by Margulis and Gromov, has lately attracted considerable interest, not least because of the lack of available results. Here we prove that the set of shapes of orders of totally real cubic fields is unbounded in the modular surface.
	\end{abstract}
	\maketitle
	
	\tableofcontents

	\section{Introduction}

	As pointed out by A.~Reznikov in the Mathoverflow conversation~\cite{mathoverflow}, Margulis and Gromov have contemplated the question of the distribution of unit lattices in number fields, when the latter are ordered in some dynamical, arithmetic or geometric fashion. Though the question can be formulated for arbitrary signatures, we confine ourselves to the exposition of the totally real case, in accordance with the setup of our result.

	Let $\sK$ be a totally real number field of degree $n$, $\Sigma$ the set of its complex field embeddings $\sK\to \bC$, all with image contained in $\bR$. Let \eql{eq:logemb}{\lambda\colon \sK^{\times}\to \bR^{\Sigma}\;, \quad \alpha\mapsto (\log|\sigma(\alpha)|)_{\sigma\in \Sigma}}
	be the logarithmic embedding of the unit group of $\sK$. The classical Dirichlet's unit theorem (see, for instance,~\cite[\textsection I.7 and \textsection I.12]{Neukirch}, more generally~\cite{Neukirch} serves as a reference for all the algebraic number theory used in this article) gives that the image of the unit group $\Or^{\times}$ of any order $\Or\subset \sK$ is a lattice in the trace-zero hyperplane \eq{H=\biggl\{(x_{\sigma})\in \bR^\Sigma:\sum x_{\sigma}=0  \biggr\}\;.} If we fix an isometric\footnote{With respect to the standard Euclidean metrics on both $H$ and $\bR^{n-1}$.} identification of $H$ with $\bR^{n-1}$, and rescale the resulting lattice to have unit covolume, we obtain a unimodular lattice in $\bR^{n-1}$. Choosing a different isometry $H\to \bR^{n-1}$ yields potentially a different unimodular lattice; all those are, however, isometric images of each other. Therefore, we may unambiguously associate to $\Or$ an element $\Delta_{\Or^{\times}}$ in the locally symmetric space
	\eq{\SL_{n-1}(\bZ)\bsl \SL_{n-1}(\bR)/\SO_{n-1}(\bR)\;,}
	identified, as is custom, with the set of unimodular lattices in $\bR^{n-1}$ up to orientation-preserving linear isometries;
we shall refer to $\Delta_{\Or^{\times}}$ as the \emph{shape} of the unit group $\Or^{\times}$.

There is a classical dynamical interpretation of such objects, as shapes of stabilizers of periodic torus orbits on the space of unimodular lattices. We briefly recall the connection, from which our own interest in the problem arose, in order to motivate a possible ordering of the shapes, which in turn offers a natural formulation for the question about their distribution; for the details, we refer to~\cite[\textsection 4.1.2]{ELMV} or~\cite{Shapira-Weiss}. Consider the positive diagonal subgroup $A<\SL_n(\bR)$, acting by right translations on the homogeneous space $X_n=\SL_n(\bZ)\bsl \SL_n(\bR)$, identified with the set of unimodular lattices in $\bR^n$. The $A$-orbit of a point $x\in X_n$ is compact if and only if the stabilizer $\Stab_A(x)$ is a lattice in $A$, the latter group being isomorphic as a topological group to $\bR^{n-1}$. Compact $A$-orbits arise from complete modules in totally real degree-$n$ number fields. A complete module $\sM$ in such a field $\sK$ is a free $\bZ$-submodule of rank $n$; its image under the embedding 
\eql{eq:embedding}{\sK\to \bR^{\Sigma}\;, \quad \alpha\mapsto (\sigma(\alpha))_{\sigma\in \Sigma}}
is, upon rescaling, is a unimodular lattice $\La_{\sM}$ in $\bR^{\Sigma}$, which we identify with $\bR^n$ by choice of a linear isomorphism. The $A$-orbit of $\La_M$ is compact; to see this, consider the multiplier ring $\Or_{\sM}=\{\alpha\in \sK: \alpha \sM\subset \sM \}$, which is an order in $\sK$ (equal to $\sM$ if and only if $\sM$ is itself an order). The intersection with $A$ of the image of the unit group $\Or_{\sM}^{\times}$ under the map in~\eqref{eq:embedding} (where we see every element of such image acting naturally as a diagonal matrix on $\bR^{\Sigma}\simeq \bR^{n}$) is a finite-index subgroup of the latter, for units have norm $N_{\sK/\bQ}(\alpha)\in \{\pm 1\}$. Since $\Or_M^{\times}$ contains a free $\bZ$-submodule of rank $n-1$, by Dirichlet's unit theorem, the same must hold for $\Stab_{A}(\La_{\sM})$, which is therefore a lattice. It is straightforward to verify that $\Stab_{A}(\La_{\sM})$ is commensurable with the lattice $\lambda(\Or_{\sM}^{\times})$ arising from the logarithmic embedding in~\eqref{eq:logemb}. We prefer to discuss shapes of the latter, instead of the former, as it is more natural from a number-theoretic perspective. 

It is a fact that every compact $A$-orbit in $X_n$ arises in the previously explicated manner; to be precise, there is a bijective correspondence between the set of compact $A$-orbits in $X_n$ and isomorphism classes of triples $(\sK,[\sM],\Phi)$, where $\sK$ is a totally real number field of degree $n$, $[\sM]$ is an equivalence class of complete modules in $\sK$ under the equivalence relation
\eq{\sM_1\sim \sM_2 \Longleftrightarrow \text{ there is $\alpha\in \sK^{\times}$ such that }\alpha\sM_1=\sM_2\;,}
and $\Phi\colon \bR^{\Sigma}\to \bR^n$ is a linear isomorphism. Two such triples $(\sK_1,[\sM_1],\Phi_1),(\sK_2,[\sM_2],\Phi_2)$ are isomorphic, naturally, is there is a field isomorphism $\sK_1\to \sK_2$ sending $\sM_1$ to $\sM_2$ and making the obvious diagram with $\Phi_1$ and $\Phi_2$ commute.

The choice of a Haar measure $m_A$ on $A$ allows to speak of the volume $\vol(A\cdot x)$ of a compact $A$-orbit with respect to $m_A$, which is the $m_A$-measure of any fundamental domain for $\Stab_A(x)$ in $A$. Up to a uniform multiplicative constant, the volume of $A\cdot \La_{\sM}$ is the regulator of the order $\Or_{\sM}$ for any complete module $\sM\subset\sK$. For any $T>0$, one can then consider the finite collection
\eq{\{ C \text{ compact $A$-orbit in }X_n: \vol(C)\leq T  \}}
(for the finiteness claim, see~\cite{Oh} or~\cite{ELMV}); the uniform probability measure on this collection gives a finitely supported Borel probability measure $\rho_T$ on $\SL_{n-1}(\bZ)\bsl \SL_{n-1}(\bR)/\SL_{n-1}(\bR)$ via pushforward by the map $A\cdot \La_{\sM}\to \Delta_{\Or_{\sM}^{\times}}$. 
\begin{rmk}
	For two alternative, equally meaningful ways of ordering compact $A$-orbits, or more generally isomorphism classes of orders for number fields of arbitrary signature and fixed degree, see~\cite{mathoverflow}.
\end{rmk}

One can then ask:
\begin{qst}(Asymptotic distribution of shapes)
	What are the limit points, in the weak$^*$ topology, of the family $(\rho_T)_{T>0}$ as $T\to\infty$?
\end{qst}
From~\cite{mathoverflow} we learned that Margulis expects most of the mass of $\rho_T$ to escape to infinity as $T\to\infty$, with respect to geometric ordering given by the volume. On the other hand, choosing the \emph{arithmetic ordering} by the discriminant (see~\cite{mathoverflow}), the conjecture is that the associated new family $(\rho'_T)_{T>0}$ of finitely supported probability measures should equidistribute in $\SL_{n-1}(\bZ)\bsl \SL_{n-1}(\bR)/\SO_{n-1}(\bR)$, that is, it should converge to the uniform probability measure $\mu$ on the latter space, which is the projection of the unique $\SL_{n-1}(\bR)$-invariant Borel probability measure on the finite-volume quotient $\SL_{n-1}(\bZ)\bsl \SL_{n-1}(\bR)$.

To date, no promising approach to gain a full distributional understanding appears to have emerged\footnote{Here the related work of Bhargava and Harron~\cite{Bhargava-Harron} on the additive counterpart of the question deserves mention: it addresses the related question of the distribution of (natural orthogonal projections of) shapes of \emph{integer lattices} in cubic, quartic and quintic number fields, establishing equidistribution for those with Galois group given by the symmetric group.}, to the best of the authors' knowledge; yet, one can aim for a coarser, topological understanding of the set of shapes. David and Shapira have initiated this study in~\cite{David-Shapira}, in connection with the problem of understanding the topology of the set of $A$-invariant end ergodic probability measures on $\SL_{n}(\bZ)\bsl \SL_{n}(\bR)$ for $n\geq 3$, and formulated a series of three conjectures~\cite[Conjecture 1.1]{David-Shapira} about the topological nature of the set of shapes of unit lattices in cubic fields (that is, the case $n=3$), culminating in the claim that the latter set is dense in $\SL_{2}(\bZ)\bsl \SL_2(\bR)/\SO_2(\bR)$. 

Our result solves the first of these conjectures, corresponding to the weaker claim that the set of shapes is unbounded. We identify $\SL_2(\bR)/\SO_2(\bR)$ with the hyperbolic plane $\bH=\{x+iy\in \bC:y>0  \}$ via the transitive $\SL_2(\bR)$-action on $\bH$ by Möbius transformations
\eq{\begin{pmatrix}a&b\\ c&d\end{pmatrix}\cdot z=\frac{az+b}{cz+d}\;,}
with $\SO_2(\bR)$ being the stabilizer of $i$. Shapes of unit groups of orders of totally real cubic fields become thus elements of the modular surface $\SL_2(\bZ)\bsl \bH$.

\begin{thm}
	\label{thm:main}
	The collection of shapes of unit groups of orders of totally real cubic number fields is unbounded in the modular surface: more precisely, the closure of the set 
	\eq{\{ \Delta_{\cO^{\times}}:\cO \text{ order of a totally real cubic number field $\sK$}  \}}
	in $\SL_2(\bZ)\bsl \bH$ is not compact.
\end{thm}

In fact, the argument shall provide a smooth curve contained in the closure of the set of shapes, which escapes to the cusp of the modular surface. The proof of Theorem~\ref{thm:main} is given in \textsection\ref{sec:proof}.

While the authors of the present manuscript were working on the remaining two conjectures, the preprint~\cite{Dang-Li} has appeared with resolution of the third (and stronger) one by Dang, Gargava and Li\footnote{The three authors manage to attain density even on the level of the unit tangent bundle $\SL_2(\bZ)\bsl \SL_2(\bR)$.}; we also refer to \emph{loc.~cit.} for further elaborations upon the importance of the question considered in this article. In the next subsection we present a brief comparison of the two approaches. 

\subsection{An outline of the proof}
The argument draws inspiration from the line of reasoning in~\cite{David-Shapira}, where it is shown that the closure of the set of shapes in the modular surface contains a countable collection of "low-lying" smooth curves. In order to locate shapes in a way which allows to pinpoint accumulation points, we confine our attention to cubic orders of the form $\Or=\bZ[\theta]$, $\theta$ being a root of a cubic monic irreducible polynomial $f(X)\in \bZ[X]$ with only real roots, with the property that the unit group $\Or^{\times}$ admits a particular fundamental system of linear units (the definition is recalled in \textsection\ref{subsec:linearunits}), of the form $(\theta,a\theta-b)$ for fixed (necessarily coprime) integers $a,b$. This class of orders appears already in works of Grundman~\cite{Grundman}, Minemura~\cite{Minemura-first,Minemura-second} and Thomas~\cite{Thomas}, and polynomials giving rise to such orders have already been fully described in~\cite{David-Shapira}; they all have the form
\eql{eq:polynomials}{f_{a,b,t}(X)=f_0(X)+tX(aX-b)}
for a fixed cubic monic $f_0(X)\in \bZ[X]$, where $t$ ranges over the integers. As $t$ tends to $+\infty$, there are two families of positive roots $(\theta_t^{(1)}),(\theta_t^{(2)})$ of $f_{a,b,t}(X)$ tending, respectively, to $0$ and $b/a$, which are the roots of the second-degree summand in~\eqref{eq:polynomials}. Interestingly, letting $t\to+\infty$ for \emph{fixed} $a,b$, one can show by a direct simplification of our arguments that the shapes of the orders $\Or_t=\bZ[\theta_t^{(1)}]$ converge, in $\SL_{2}(\bZ)\bsl \bH$, to the primitive sixth root of unity $\frac{1}{2}+i\frac{\sqrt{3}}{2}$. 

Instead, we let the parameter $a$ vary as a function of $t$, and set $b=1$\footnote{The converse case of $a=1$ and $b$ varying with $t$ leads to the so-called \emph{simplest cubic fields} (from the ease of computation of their algebraic invariants), already investigated in works of Cassels~\cite{Cassels}, Ankeny-Brauer-Chowla~\cite{ABC}, Cusick~\cite{Cusick} and Washington~\cite{Washington}.}. Making essential use of quantitative error estimates on the distances $|\theta_t^{(1)}|$, $|\theta_t^{(2)}-a_t^{-1}|$, and judiciously choosing the growth rate of $a_t$, which we set to be $t^{\alpha}$ for small real $\alpha>0$, we obtain an explicit smooth curve, parametrized by $\alpha$, of accumulation points for the shapes of the orders $\Or_t$, which we show leaves every compact set in the modular surface.

Whereas our overarching strategy is definitely akin in spirit, the approach of Dang, Gargava and Li to obtain density in~\cite{Dang-Li} proceeds by studying varying families of orders inside the simplest cubic fields; these collection of "extracted" orders is then shown to contain a horospherical piece and to be invariant by the action of the positive diagonal subgroup of $\SL_2(\bR)$ on $\SL_2(\bZ)\bsl \SL_2(\bR)$, a feature which, by a classical "banana-trick" argument of Margulis (going back to his thesis~\cite{Margulis}) yields density.

\begin{rmk}
	It seems possible to carry out similar computations and estimates for orders $\bZ[\theta]$ with linear fundamental units of the more general form $(a\theta-b,c\theta-d)$ where $a,b,c,d\in \bZ$, $ad-bc\neq 0$ and the pairs $(a,b),(c,d)$ are coprime. The case $ad-bc=1$ had been previously characterized in~\cite[Theorem 3.7]{David-Shapira}, but not exploited there for the purpose of eliciting information about the closure of the set of shapes of cubic orders.  
	
	For fixed $a,b,c,d$ as above, under further appropriate conditions, one can prove that there are countably many monic cubic irreducible polynomials $f_{a,b,c,d,t}(X)\in \bZ[X]$ giving rise to such orders, where $t$ ranges in $\bZ$. By letting $a=a_t$ and $c=c_t$ vary with $t$, in such a way that $a_t$ and $c_t$ are of the same order, respectively, of $t^{\alpha}$ and $t^{\gamma}$ for unrelated real parameters $\alpha,\gamma>0$, both close to $0$, and choosing $b_t$ and $d_t$ accordingly, it stands to reason to expect that one can attain non-emptiness of the interior of the closure of the set of shapes, perhaps even density on the modular surface. However, in light of the concurrent appearance of the density result of~\cite{Dang-Li}, we dediced not to pursue this further.
\end{rmk}

\subsection*{Notation}
If $g\colon \bZ\to \bR_{\geq 0}$ is a function, we use Landau's notation $\O(g)$ (respectively, $\Theta(g)$) to indicate a function $f\colon \bZ\to \bR$ such that there are $C>0$ and an integer $t_0\geq 1$ with $|f(t)|\leq Cg(t)$ (respectively, $C^{-1}g(t)\leq |f(t)|\leq Cg(t)$) for all $t\geq t_0$. To indicate an unspecified function $f\colon \bZ\to \bR$ which is infinitesimal as $t\to+\infty$ we use the usual symbol $\o(1)$.

\section{Cubic orders with linear fundamental units and unboundedness of shapes}
\label{sec:proof}

\subsection{Orders with fundamental systems of linear units} 
\label{subsec:linearunits}
Let $\Or$ be an order of a totally real cubic field $\sK$. The unit group $\Or^{\times}$, modulo its torsion subgroup, is known by Dirichlet's unit theorem to be a free abelian group of rank $2$. A pair $(\theta_1,\theta_2)$ of $\bZ$-independent generators for the latter group will be referred to as a \emph{system of fundamental units} for $\Or$. 

We will consider orders $\Or=\bZ[\theta]$, with $\theta$ a root of a cubic monic irreducible polynomial $f(X)$, with linear fundamental units, of the form $(\theta,a\theta-b)$ for some $a,b\in \bZ$; the shape of their unit group is thus the equivalence class of the lattice\footnote{We indicate with $\langle v_1,\dots,v_k\rangle_{\bZ}$ the $\bZ$-submodule generated by a collection of vectors $v_1,\dots,v_k$ in some Euclidean space $\bR^d$.} 
\eq{\langle (\theta, \theta^{(2)},\theta^{(3)}),(a\theta-b,a\theta^{(2)}-b,a\theta^{(3)}-b)  \rangle_{\bZ}\subset H\simeq \bR^2\;, }
where $\theta^{(2)},\theta^{(3)}$ are the two remaining roots of $f(X)$. Those orders have been fully classified in~\cite{David-Shapira}; following the terminology in \emph{loc.~cit.}, we say that a pair $(a,b)$ of integers is a \emph{mutually cubic pair} if $a^3\equiv 1 \;(\Mod b)$ and $b^3\equiv 1 \;(\Mod a)$.
	
	\begin{prop}[{\cite[Theorem 3.5]{David-Shapira}}]
		Let $(a,b)$ be a mutually cubic pair of coprime non-zero integers, and consider the family of polynomials
		\eq{f_{a,b,t}(X)=X^3+\frac{(a^3-1)^2-b^3}{ab^2}\;X^2-a\;\frac{a^3-1}{b}\;X+1+tX(aX-b)}
for $t\in \bZ$. Then, for all $t\in \bZ$, all complex roots of $f_{a,b,t}(X)$ are real; furthermore, if $f_{a,b,t}(X)$ is irreducible over $\bQ$, $\theta$ is a root of $f_{a,b,t}(X)$ and $\sK=\bQ[\theta]$, then $\theta$ and $a\theta-b$ are units in the order $\bZ[\theta]$ of the cubic field $\sK$.
	\end{prop}

	Henceforth, we shall consider a family of mutually cubic pairs $(a_t,1)$ depending on an integer parameter $t\geq 1$; to be precise, we fix, for each such $t$, an integer $a_t\geq 0$, upon which we shall place restrictions as the needs of our argument emerge, and consider the associated polynomial
	\eq{f_t(X)\coloneqq f_{a_t,1,t}(X)=X^3+a_t^2(a_t^3-2)X^2-a_t(a_t^3-1)X+1+tX(a_tX-1)\;.}

	\begin{lem}
		There exists an integer $A\geq 1$ such that, for all integers $t\geq 1$ and $a\geq A$, the polynomial $f_{a,1,t}(X)$ is irreducible over $\bQ$.
	\end{lem}
\begin{proof}
Since $f(X)\coloneqq f_{a,1,t}(X)$ is cubic, monic, with integer coefficients, and with constant term $1$, it is reducible over $\bQ$ if and only either $f(1)=0$ or $f(-1)=0$. We compute 
\eq{f(1)=a^5-a^4-2a^2+a+t(a-1)+2\geq a^{5}-a^4-2a^2+2a+1}
and 
\eq{f(-1)=a^5+a^4-2a^2-a+t(a+1)\geq a^{5}+a^4-2a^2+1\;,}
from which we deduce at once that $f(\pm1)\neq 0$ for all sufficiently large $a$, independently of $t$.
\end{proof}

\subsection{Estimates for the roots}
In the sequel, we operate under the assumption $a_t\geq A$, for $A$ as in the previous lemma. 
The derivative of the polynomial $f_t(X)$ is 
\eq{f_t'(X)=3X^2+2(a_t^2(a_t^3-2)+ta_t)X-(t+a_t(a_t^3-1))\;;}
under our working assumptions $t\geq 1$, $a_t\geq 1$, an elementary verification shows that that the roots of $f_t'(X)$ have different sign, and since $f_t(0)=1$ it follows that $f_t(X)$ admits two positive roots\footnote{Since the constant term of $f_t(X)$ is $1$, the third root is necessarily negative by Vieta's formulas; we shall not need it explicitly in our argument.} $\theta_t^{(1)}<\theta_t^{(2)}$. We will focus our attention on the orders \eq{\cO_t\coloneqq\bZ[\theta_t^{(1)}]}
of the cubic fields $\bQ[\theta_t^{(1)}]$.

Intuitively, as $f_t(X)$ is of the form $f_0(X)+tX(a_tX-1)$, it should be the case that $\theta_t^{(1)}$ and $\theta_t^{(2)}$ approximate respectively, as $t\to\infty$, the roots $0$ and $a_t^{-1}$ of the polynomial $tX(a_tX-1)$. We shall make use of appropriate quantifications of this phenomenon, relying on the following first (rather crude) estimate, to appreciate the significance of which we hasten to point out that we wiil ultimately choose $a_t$ to be of the same order of $t^{\alpha}$ for values of $\alpha>0$ tending to $0$.

\begin{lem}
	Fix $0<\alpha<1/4$, and suppose $a_t\leq t^{\alpha}$ for all $t\geq 1$. Then, for all positive integers $k\leq 1/\alpha-3$, there is $A'=A'(\alpha,k)>0$ such that 
	\eql{eq:firstbound}{|\theta_t^{(1)}|\leq a_t^{-k} \quad \text{and} \quad |\theta_t^{(2)}-a_t^{-1}|\leq a_t^{-k}}
	for all $a_t\geq A'$.
\end{lem}
\begin{proof}
	We compute, for all integers $k\geq 1$, 
	\eq{f_t(a_t^{-k})=\frac{1}{a_t^{3k}}(-ta_t^{2k}+ta_t^{k+1}+a_t^{3k}-a_t^{2k+4}+a_t^{2k+1}+a_t^{k+5}-2a_t^{k+2})\;.}
	In view of the bounds assumed on $\alpha$ and $k$, the leading term of the numerator on the right-hand side is $-ta_t^{2k}$; it follows that there is some $A'_1=A'_1(\alpha,k)>0$ such that $f_t(a_t^{-k})<0$ for all $a_t\geq A_1'$. Since $f_t(0)>0$, the bound $|\theta_t^{(1)}|\leq a_t^{-k}$ is a consequence of the intermediate value theorem.
	
	The argument for the second bound $|\theta_t^{(2)}-a_t^{-1}|\leq a_t^{-k}$ runs along similar lines, by realizing that, for all $a_t\geq A'_2=A'_2(\alpha,k)$, the inequalities $f_t(a_t^{-1})>0$ and $f_t(a_t^{-1}-a_t^{-k})<0$ hold, and applying the intermediate value theorem to locate the root $\theta_t^{(2)}$ in the open interval $(a_t^{-1}-a_t^{-k},a_t^{-1})$.
\end{proof}

Hereinafter, assume $\alpha\in \bR$ and $k\in \bN^*$ are fixed so that $0<\alpha<1/4$ and $k\leq 1/\alpha-3$; we select $a_t$ in such a way that 
\eq{\sup\{A,A' \}\leq ct^{\alpha}\leq a_t\leq t^{\alpha}}
for some $0<c<1$, where $A'=A'(\alpha,k)$ is given by the last lemma.

We now give more accurate error estimates for $|\theta_t^{(1)}|$ and $|\theta_t^{(2)}-a_t^{-1}|$ via Taylor's approximation. The Taylor expansion of $f_t(X)$ at $0$ gives
\eql{eq:firstTaylor}{0=f_t(\theta_t^{(1)})=f_t(0)+f_t'(0)\theta_t^{(1)}+\delta_t^{(1)}=1-(t+a_t(a_t^3-1))\theta_t^{(1)}+\delta_t^{(1)}}
with
\eql{eq:firstbounderror}{|\delta_t^{(1)}|\leq \frac{1}{2}|f''_t(0)||\theta_t^{(1)}|^2+\frac{1}{6}|f_t^{'''}(0)||\theta_t^{(1)}|^3\leq (ta_t+a_t^5-2a_t^2)a_t^{-2k}+ a_t^{-3k}\;,}
using~\eqref{eq:firstbound} in the last inequality.
We choose the value of $k$ achieving the optimal bound in~\eqref{eq:firstbound}, namely we impose $k\geq 1/\alpha -4$. With such choices, we have from~\eqref{eq:firstbounderror} that 
\eql{eq:betterfirstbound}{|\delta_t^{(1)}|\leq  (c^{-1/\alpha}a_t^{1+1/\alpha}+a_t^5-2a_t^2)a_t^{-2k}+ a_t^{-3k}= \Theta(a_t^{9-1/\alpha})\;;}
from~\eqref{eq:firstTaylor}, we write
\eql{eq:firstgap}{\theta_t^{(1)}=\frac{1+\delta_t^{(1)}}{t+a_t(a_t^{3}-1)}=\frac{1}{t}\;\eps_t^{(1)}}
with $|\eps_t^{(1)}|=\Theta(1)$, and thus
\eql{eq:firstlimit}{\liml_{t\to\infty}\frac{\log{|\eps^{(1)}_t|}}{\log{t}}=0\;,}
as follows readily from~\eqref{eq:betterfirstbound} and the relation between $t$ and $a_t$. 

Similarly, the Taylor expansion of $f_t(X)$ at $a_t^{-1}$ yields
\eql{eq:secondTaylor}{0=f_t(\theta_t^{(2)})=f_t(a_t^{-1})+f'_t(a_t^{-1})(\theta_t^{(2)}-a_t^{-1})+\delta_t^{(2)}}
with
\eq{|\delta_t^{(2)}|\leq \frac{1}{2}|f''(a_t^{-1})||\theta_t^{(2)}-a_t^{-1}|^2+|\theta_t^{(2)}-a_t^{-1}|^3=\Theta(|a_t|^{5-1/\alpha})\;,}
the last estimate following from ~\eqref{eq:firstbound} and the fact that $|f''(a_t^{-1})|=\Theta(a_t^{1+1/\alpha})$. From~\eqref{eq:secondTaylor} we write 
\eql{eq:secondgap}{a_t\theta_t^{2}-1=-a_t\;\frac{f_t(a_t^{-1})+\delta_t^{(2)}}{f_t'(a_t^{-1})}=-a_t\;\frac{a_t^{-3}+\delta_t^{(2)}}{t+a_t(a_t^3-3)+3a_t^{-2}}=\frac{1}{t}\;\eps_t^{(2)}}
where it is easily seen that $|\eps_t^{(2)}|=\Theta(t^{-(3\alpha-1)})$, so that
\eql{eq:secondlimit}{\liml_{t\to\infty}\frac{\log|\eps_t^{(2)}|}{\log{t}}=1-3\alpha\;.}

We shall make essential use of~\eqref{eq:firstlimit} and~\eqref{eq:secondlimit} in the proof of the forthcoming proposition, as well as directly in the proof of Theorem~\ref{thm:main}.

\subsection{A system of fundamental units} Armed with the estimates of the foregoing subsection, we are now ready to prove that, with our choices of parameters, the pair $(\theta_t^{(1)},a_t\theta_t^{(1)}-1)$ is a system of fundamental units for the order $\Or_t=\bZ[\theta_t^{(1)}]$, at least when $t$ is sufficiently large.

\begin{prop}
	\label{prop:fundunits}
	Let $0<\alpha<1/4$, and let $\Or_t=\bZ[\theta_t^{(1)}]$ be the order defined above for all integers $t\geq 1$, with the coefficients $a_t$ satisfying $ct^{\alpha}\leq a_t\leq t^{\alpha}$ for some $0<c<1$. Then, for all sufficiently large $t$ (depending only on $\alpha$), $(\theta_t^{(1)},a_t\theta_t^{(1)}-1)$ is a system of fundamental units for $\cO_t$.
\end{prop}

The proof relies on the following inequality between the discriminant and the regulator of a totally real cubic field, established by Cusick in~\cite{Cusick}.
\begin{prop}
	\label{prop:Cusick}
	Let $\sK$ be a totally real cubic number field, $\Or$ an order of $\sK$ of discriminant $D$ and regulator $R$. Then 
	\eq{\frac{R}{\log^2(D/4)}\geq \frac{1}{16}\;.}
\end{prop}
Cusick's result if formulated for the maximal order of a totally real cubic field $\sK$, but the proof is easily adapted to any order of $\sK$. Proposition~\ref{prop:Cusick} implies, in particular, that in order to show $(\theta_t^{(1)},a_t\theta_t^{(1)}-1)$ is a fundamental system of units for $\Or_t$, it suffices to establish the inequality
\eql{eq:bounds}{0\neq \frac{R_t'}{\log^2(D_t/4)}<\frac{1}{8}}
for $D_t$ the discriminant of $\Or_t$ and $R'_t$ the relative regulator of $(\theta_t^{(1)},a_t\theta_t^{(1)}-1)$; the straightforward argument for this is presented in~\cite[Corollary 2.2]{David-Shapira}.

\begin{proof}[Proof of Proposition~\ref{prop:fundunits}]
	The discriminant $D_t$ of $\Or_t$ is 
	\eq{\begin{vmatrix} 1 & \theta_t^{(1)} & (\theta_t^{(1)})^2 \\ 1 & \theta_t^{(2)} & (\theta_t^{(2)})^2 \\ 1 & \theta_t^{(3)} & (\theta_t^{(3)})^{2} \end{vmatrix}^2=(\theta_t^{(2)}-\theta_t^{(1)})^2(\theta_t^{(3)}-\theta_t^{(1)})^2(\theta_t^{(3)}-\theta_t^{(2)})^2\;.}
From~\eqref{eq:firstgap},~\eqref{eq:firstlimit},~\eqref{eq:secondgap},~\eqref{eq:secondlimit} and Vieta's formulas we find that
	\eq{\theta_t^{(1)}=\frac{\eps_t^{(1)}}{t}\;, \quad \theta_t^{(2)}=\frac{1}{a_t}+\frac{\eps_t^{(2)}}{ta_t}=\Theta(t^{-\alpha})\;, \quad \theta_t^{(3)}=-\frac{1}{\theta_t^{(1)}\theta_t^{(2)}}=\Theta\biggl(\frac{t^{1+\alpha}}{\eps_t^{(1)}}\biggr)\;.}
	These lead to the estimate
\eq{D_t= \Theta\biggl(\frac{t^{1+\alpha}}{\eps_t^{(1)}}\biggr)^{4}\;.}
We now turn to estimating the relative regulator $R_t'$, which equals
\eqs{\begin{vmatrix} \log|\theta_t^{(1)}| & \log|a_t\theta_t^{(1)}-1| \\ \log|\theta_t^{(2)}| & \log|a_t\theta_t^{(2)}-1|   \end{vmatrix}&=\log|\theta_t^{(1)}|\log|a_t\theta_t^{(2)}-1|-\log|\theta_t^{(2)}|\log|a_t\theta_t^{(1)}-1|\\
&=(-\log{t}+\log|\eps_t^{(1)}|)(-\log{t}+\log|\eps_t^{(2)}|)+\o(1)\\
&=\log^2t\biggl(1-\frac{\log|\eps_t^{(1)}|}{\log t} -\frac{\log|\eps_t^{(2)}|}{\log t}\biggr)+\log|\eps_t^{(1)}|\log|\eps_t^{(2)}|+o(1)\\
&= \Theta(\log^2(t^{3\alpha}))\;.}
Given our bounds on $\alpha$, it is now apparent that~\eqref{eq:bounds} holds for all $t$ sufficiently large, purely in terms of $\alpha$.
\end{proof}

\subsection{Proof of Theorem~\ref{thm:main}}
We work with the orders $\Or_t=\bZ[\theta_t^{(1)}]$ defined as above, keeping with the choice of parameters (and with the notation) we already made. Let $\theta_t^{(3)}<0$ be the third root of $f_t(X)$. The shape of the unit group $\Or_{t}^{\times}$ is the equivalence class, in $\SL_2(\bZ)\bsl \bH$, of the lattice 
\eq{\langle (\log|\theta_t^{(1)}|,\log{|\theta_t^{(2)}|},\log{|\theta_t^{(3)}|}),(\log|a_t\theta_t^{(1)}-1|,\log|a_t\theta_t^{(2)}-1|,\log|a_t\theta_t^{(3)}-1|) \rangle_{\bZ}\subset H\;,}
identified isometrically with a lattice in $\bR^2$. As shapes are invariant under scaling, we can identify the previous lattice in $H$ with the lattice 
\eq{\biggl\langle \biggl(-\log{|\theta_t^{(1)}|}-\frac{1}{2}\log|\theta_t^{(2)}|,-\frac{\sqrt{3}}{2}\log|\theta_t^{(2)}|\biggr), \biggl(-\log|a_t\theta_t^{(1)}-1|-\frac{1}{2}\log{|a_t\theta_t^{(2)}-1|},-\frac{\sqrt{3}}{2}\log|a_t\theta_t^{(2)}-1|\biggr)  \biggr\rangle_{\bZ}}
in $\bR^2$, via the similarity map 
\eq{H\to \bR^2\;, \quad (-1,0,1)\mapsto (1,0)\;, \;(0,-1,1)\mapsto (1/2,\sqrt{3}/2)\;,}
identifying $H$  with the trace-zero hyperplane in $\bR^{3}$ via the ordering $\sigma_i(\theta^{(1)})=\theta^{(i)}$, $i=1,2,3$, of the embeddings.
For ease of notation, let 
\eq{v_t=\biggl(-\log{|\theta_t^{(1)}|}-\frac{1}{2}\log|\theta_t^{(2)}|,-\frac{\sqrt{3}}{2}\log|\theta_t^{(2)}|\biggr)\;,} \eq{w_t=\biggl(-\log|a_t\theta_t^{(1)}-1|-\frac{1}{2}\log{|a_t\theta_t^{(2)}-1|},-\frac{\sqrt{3}}{2}\log|a_t\theta_t^{(2)}-1|\biggr)}
be the two given $\bZ$-independent generators of the lattice. Combining~\eqref{eq:firstbound},~\eqref{eq:firstgap} and~\eqref{eq:secondgap}, we can write
\eq{v_t=\log{t}\biggl(1-\frac{\log|\eps_1^{(t)}|}{\log{t}}-\frac{1}{2}\frac{\log|a_t^{-1}+\O(a_t^{4-1/\alpha})|}{\log{t}},-\frac{\sqrt{3}}{2}\frac{\log|a_t^{-1}+\O(a_t^{4-1/\alpha})|}{\log{t}}\biggr)=\log{t}\;v_t'\;,}
\eq{w_t=\log{t}\biggl(\frac{1}{2}-\frac{1}{2}\frac{\log{|\eps_t^{(2)}|}}{\log{t}}-\frac{\log{|a_t\O(a_t^{4-1/\alpha})-1|}}{\log{t}},\frac{\sqrt{3}}{2}-\frac{\sqrt{3}}{2}\frac{\log{|\eps_t^{(2)}|}}{\log{t}}\biggr)=\log{t}\;w_t'\;,}
with 
\eq{v_t'\to\biggl(1+\frac{1}{2}\alpha,\frac{\sqrt{3}}{2}\alpha\biggr)\;, \quad w_t'\to \biggl(\frac{3\alpha}{2},\frac{3\sqrt{3}\alpha}{2}\biggr)}
as $t\to\infty$, by virtue of~\eqref{eq:firstlimit} and~\eqref{eq:secondlimit}. In the limit as $t\to\infty$, we thus obtain the equivalence class of 
\eq{\biggl\langle \biggl(\frac{1+\frac{1}{2}\alpha}{3\alpha},\frac{\sqrt{3}}{6}\biggr),\biggl(\frac{1}{2},\frac{\sqrt{3}}{2}\biggr) \biggr\rangle_{\bZ}}
in $\SL_2(\bZ)\bsl \bH$. It is well known (see, for instance, Serre's treatment in~\cite[\textsection VII.2]{Serre}) that this corresponds to the $\SL_2(\bZ)$-equivalence class of $\omega_2/\omega_1\in \bH$ for 
\eq{\omega_1=\frac{1+\alpha/2}{3\alpha}+i\frac{\sqrt{3}}{6}\;, \quad \omega_2=\frac{1}{2}+i\frac{\sqrt{3}}{2}\;;}
letting $\alpha'=\frac{1+\alpha/2}{3\alpha}$, we compute 
\eq{\frac{\omega_2}{\omega_1}=\frac{\alpha'/2+1/4}{\alpha'^2+1/12}+i\frac{\sqrt{3}\alpha'/2-\sqrt{3}/12}{\alpha'^2+1/12}\;.}
As $\alpha$ tends to $0$, the ratio $\omega_2/\omega_1$ tends to $0$, so that its $\SL_2(\bZ)$-equivalence class in the modular surface tends to the cusp.

\end{document}